\documentclass[12pt]{amsart}

\def\>{\relax\ifmmode\mskip.666667\thinmuskip\relax\else\kern.111111em\fi}
\def\<{\relax\ifmmode\mskip-.333333\thinmuskip\relax\else\kern-.0555556em\fi}
\def\vsk#1>{\vskip#1\baselineskip}
\def\vv#1>{\vadjust{\vsk#1>}\ignorespaces}
\def\vvn#1>{\vadjust{\nobreak\vsk#1>\nobreak}\ignorespaces}

\newtheorem*{corr}{Corollary of Theorems 1.4 and 1.7}

\newtheorem{thm}{Theorem}[section]
\newtheorem{cor}[thm]{Corollary}
\newtheorem{lem}[thm]{Lemma}

\newtheorem{Rem}[thm]{Remark}

\theoremstyle{definition}                                  
\numberwithin{equation}{section}

\theoremstyle{definition}

\numberwithin{equation}{section}

\usepackage{amssymb}
\usepackage{comment}
\usepackage[all]{xy}
\usepackage{enumerate}

\usepackage[all]{xy}
\usepackage{bbm}

\usepackage[dvipsnames]{xcolor}
\usepackage{graphicx}

\usepackage[colorlinks=true,allcolors = blue]{hyperref} 

\usepackage{amsmath}
\usepackage{amsxtra}
\usepackage{amscd}
\usepackage{amsthm}
\usepackage{amsfonts}
\usepackage{amssymb}
\usepackage{eucal}
\usepackage{verbatim}
\usepackage[all]{xy}
\usepackage{mathrsfs}
\usepackage{lineno}
\usepackage{color}
\definecolor{deepjunglegreen}{rgb}{0.0, 0.29, 0.29}
\definecolor{darkspringgreen}{rgb}{0.09, 0.45, 0.27}

\usepackage{stmaryrd}

\usepackage{etex,etoolbox,needspace}
\pretocmd\section{\Needspace*{4\baselineskip}}{}{}

\usepackage{tikz-cd}

  \newcommand{\ga}{\gamma}
\newcommand{\Ga}{\Gamma}  
 
 \newcommand{\ka}{\kappa}

\newcommand{\om}{\omega} \newcommand{\Om}{\Omega}

\newcommand{\bF}{{\mathbb F}}

\newcommand{\bK}{{\mathbb K}}

\newcommand{\bZ}{{\mathbb Z}}

\newcommand{\nc}{\newcommand}

\nc\wh{\widehat}

\nc\on{\operatorname}

\nc\Gr{\on{Gr}}

\nc\Fl{\on{Fl}}

 \DeclareMathOperator{\Spf}{{Spf}}

\newcommand{\limto}{{\displaystyle\lim_{\longrightarrow}}}
\newcommand{\rightlim}{\mathop{\limto}}

\newcommand{\leftlim}{\mathop{\displaystyle\lim_{\longleftarrow}}}
\newcommand{\limfromn}{\leftlim\limits_{\raise3pt\hbox{$n$}}}
\newcommand{\limton}{\rightlim\limits_{\raise3pt\hbox{$n$}}}

\newcommand{\rightlimit}[1]{\mathop{\lim\limits_{\longrightarrow}}\limits%
                    _{\raise3pt\hbox{$\scriptstyle #1$}}}

\newcommand{\leftlimit}[1]{\mathop{\lim\limits_{\longleftarrow}}\limits%
                    _{\raise3pt\hbox{$\scriptstyle #1$}}}

\makeatletter

\newcommand{\Rmnum}[1]{\expandafter\@slowromancap\romannumeral #1@}
\makeatother

\theoremstyle{definition}

\def\vsk#1>{\vskip#1\baselineskip}
\def\<{\relax\ifmmode\mskip-.333333\thinmuskip\relax\else\kern-.0555556em\fi}

\nc{\bea}{\begin{eqnarray*}}
\nc{\eea}{\end{eqnarray*}}
\nc{\bean}{\begin{eqnarray}}
\nc{\eean}{\end{eqnarray}}

\let\bi\bibitem

\let\der\partial
\let\mc\mathcal
\def\F{{\mathbb F}}
\def\C{{\mathbb C}}
\def\ox{{\otimes}}
\def\Z{\mathbb Z}

\renewcommand{\thefootnote}{[\arabic{footnote}]}

\def\K{\mathbb K}

\def\Spf{{\on{Spf}}}

\begin{document}

\title[On $p$-adic solutions to KZ equations]
{On $p$-adic solutions to   KZ  equations, ordinary crystals, and $p^s$-hypergeometric solutions}
 \author[Alexander Varchenko]{Alexander Varchenko$^{\star}$}
\author[Vadim Vologodsky]{Vadim Vologodsky$^\circ$}

\maketitle

\begin{center}

{\it $^{\star}$Department of Mathematics, University
of North Carolina at Chapel Hill\\ Chapel Hill, NC 27599-3250, USA\/}

\vsk.5>

{\it ${}^\circ$ Department of Mathematics, Princeton University \\ Princeton,
NJ 08540, USA\/}

\end{center}

{\let\thefootnote\relax
\footnotetext{\vsk-.8>\noindent
{$^\star$\sl E-mail}:\enspace anv@email.unc.edu
\\
{$^\circ$\sl E-mail}:\enspace vologod@gmail.com,
supported by a Simons Foundation Investigator Grant No. 622511 through
Bhargav Bhatt.
}}

\begin{abstract}

In this paper we consider  the KZ connection associated with a family 
of hyperelliptic curves of genus $g$ over the 
 ring of $p$-adic integers $\Z_p$.  Then the dual connection
 is the Gauss-Manin connection of that family. We observe that the Gauss-Manin connection
 has an ordinary $F$-crystal structure  and its  unit root subcrystal is of rank $g$. 
We prove that all local flat sections of the KZ connection annihilate the unite root subcrystal,
and the space of all local flat sections of the KZ connection is a free $\Z_p$-module of rank $g$.

 We also consider the reduction modulo $p^s$ of the unit root subcrystal for any $s\geq 1$.
  We prove that its annihilator  is generated by the so-called 
$p^s$-hypergeometric flat sections of the KZ connection.  In particular, that means that
 the reduction 
modulo $p^s$ of an arbitrary local flat section of the KZ connection over $\Z_p$ is a linear combination
of the $p^s$-hypergeometric flat sections.

\end{abstract}

{\small \tableofcontents  }

\newpage

\section{Introduction}
\label{sec 1}

The Knizhnik-Zamolodchikov (KZ)  
equations were discovered in conformal field theory, see \cite{KZ, EFK}. 
Versions of these equations have been observed 
  in representation theory, algebraic geometry, combinatorics, field theory, \dots\
In \cite{CF,DF,DJMM,SV1} the KZ equations were identified with equations for flat sections of suitable Gauss-Manin connections, and solutions to the KZ equations were constructed in the form of multidimensional hypergeometric integrals.

\smallskip

In \cite{SV2,V3}, the KZ equations were considered modulo $p^s$ where $p$ is a prime number,
and polynomial solutions to the KZ equations modulo $p^s$ were constructed by an elementary 
procedure as suitable $p^s$-approximations of the hypergeometric integrals.

\smallskip

In \cite{VV}, the question was addressed of whether all solutions to the KZ equations in characteristic $p$ are generated by the $p$-hypergeometric functions. The affirmative answer was demonstrated in a non-trivial example.

\smallskip

In \cite{V3,VZ},  the KZ equations associated with a family of hyperelliptic curves of genus $g$ were studied. 
It was demonstrated that the module of $p^s$-hypergeometric solutions has a $p$-adic limit as $s$ tends to infinity. 
This limit is a KZ-invariant rank $g$ subbundle $\mc W$ of the associated KZ connection.

\smallskip

In this paper we describe how the  $p^s$-hypergeometric solutions and  subbundle $\mc W$ are related to the ordinary $F$-crystal 
of the Gauss-Manin connection of that family of hyperelliptic curves.  
More precisely,  we consider  the KZ connection associated with the family 
of hyperelliptic curves  over the 
 ring of $p$-adic integers $\Z_p$.  Then the dual connection
 is the Gauss-Manin connection of that family. We observe that the Gauss-Manin connection
 has an ordinary $F$-crystal structure  and its  unit root subcrystal is of rank $g$. 
We prove that all local flat sections of the KZ connection annihilate the unite root subcrystal,
 the space of all local flat sections of the KZ connection is a free $\Z_p$-module of rank $g$,
 the space of all local flat sections generate the subbundle $\mc W$.

\smallskip

 We also consider the reduction modulo $p^s$ of the unit root subcrystal for any $s\geq 1$.
  We prove that its annihilator  is generated by the 
$p^s$-hypergeometric flat sections of the KZ connection.  In particular, that means that
 the reduction 
modulo $p^s$ of an arbitrary local flat section of the KZ connection over $\Z_p$ is a linear combination
of the $p^s$-hypergeometric solutions. 

\smallskip

The general project is to analyze solutions of arbitrary KZ equations over 
$\Z_p$ and relate them to $p^s$-hypergeometric solutions and associated 
$F$-crystals. At this stage, we consider first non-trivial examples to establish a map 
of the interrelations between these concepts.

\subsection{KZ equations}
\label{sec 1.1}

Let $g$ be a positive integer. Denote $n=2g+1$.

In this paper we consider the following example of the KZ equations.
Let $\Om_{ij}$ be the $n\times n$-matrix $(c_{kl})$ with 
$c_{ij} = c_{ji}=-c_{ii}= -c_{jj}=1$ and all other entries equal to $0$.
Let $z=(z_1,\dots,z_n)$. 
Define 
\bea
H_i(z) = \sum_{j \ne i}
   \frac{\Omega_{ij}}{z_i - z_j}  ,
\quad i = 1, \dots , n,
\eea
the Gaudin Hamiltonians.
Consider the system of differential and algebraic equations for 
 column vectors  $I(z)=(I_1(z)$, \dots, $I_{n}(z))^\intercal$:
 \bean
\label{KZ}
\phantom{aaaaaa}
2\der_i I \ - \
   \,H_i(z) \,I \,=\,0,
\quad i = 1, \dots , n,
\quad
I_1(z)+\dots+I_{n}(z)=0,
\eean
where $\der_i$ denotes $\frac{\partial }{\partial z_i}$; \
$M^\intercal$ denotes the transpose matrix of a matrix $M$.

\smallskip

Let $\mathbb A^n$ be affine space with coordinates $z_1,\dots,z_n$.
Let
$S$ be the complement to the union of all diagonal hyperplanes defined by equations
$z_i=z_j$, $i\ne j$.
The KZ equations define a flat KZ connection $\nabla^{\on{KZ}}$
on the trivial bundle $S\times V\to S$ whose fiber is
the vector space $V$ of $n$-vectors with the zero sum of coordinates.
Denote this bundle by $\mc V$ and call the pair $(\mc V, \nabla^{\on{KZ}})$ the KZ bundle.

\begin{Rem}

The system of equations \eqref{KZ} is the system of the original KZ equations 
associated with the Lie algebra $\frak{sl}_2$ and the subspace of singular vectors of weight $n-2$
of the tensor power $(\C^2)^{\ox n}$  of the two-dimensional irreducible $\frak{sl}_2$-module, up to a gauge
transformation, see this example in \cite[Section 1.1]{V1}. The original KZ equations depend on a parameter $\ka$. 
In our paper,  $\ka=2$, see the coefficient of $\der_i$ in \eqref{KZ}.

\end{Rem}

\subsection{Gauss-Manin connection}

Consider the family  $X$ of hyperelliptic curves $X(z)$ defined by the affine equation
\bean
\label{hE}
y^2=(x-z_1)\dots(x-z_n).
\eean
For $z\in S$, the curve $X(z)$ is  a smooth projective curve of genus $g$. 
The KZ connection $\nabla^{\on{KZ}}$
 is dual  to the Gauss-Manin connection  $\nabla^{\on{GM}}$
over $S$ with fiber  $H^1_{dR}(X(z))$.

\smallskip

For example,  the flat sections of the  KZ equations over the field $\C$ are labeled by flat sections
$\ga(z)$ of the Gauss-Manin connection over $S$ with fiber $H_1(X(z),\C)$
and  given by the following construction.
 Define the master function
\bean
\label{mast f}
\Phi(x,z) = \big((x-z_1)\dots(x-z_n)\big)^{-1/2}
\eean
and  the ${n}$-vector  of hyperelliptic   integrals
\bean
\label{Iga}
I^{(\ga)} (z)=(I_1(z),\dots,I_n(z))^\intercal,\qquad I_j=
\int_{\ga(z)} \frac{dx}{(x-z_j)y}
=
\int_{\ga(z)} \frac{\Phi(x,z)dx}{x-z_j}\,.
\eean

\begin{thm}
\label{thm1.1}

 The vector $I^{(\ga)}(z)$ satisfies  the KZ equations \eqref{KZ}.
All complex solutions to  the KZ equations  have this form.
The complex vector space of solutions of the form \eqref{Iga} is $n-1$-dimensional.

\end{thm}

Theorem \ref{thm1.1} is a classical statement discussed as an example in  \cite[Section 1.1]{V1}, also see
\cite[Formula (1.3)]{V1}.

 The solutions in \eqref{Iga} are called the hypergeometric solutions to the KZ equations.

\subsection{Solutions modulo $p^s$}

Let $p$ be an odd prime number.   For an integer $s\geq 1$ define the {\it master polynomial}
\bea
\Phi_s(x,z) = \big((x-z_1)\dots(x-z_n)\big)^{(p^{s}-1)/2}
\eea
which can be considered as a $p^s$-approximation of the master function $\Psi(x,z)$.
 For $\ell=1,\dots, g$, define the column $n$-vector
\bea
Q^{s,\ell}(z) = (Q^{s,\ell}_1,\dots, Q^{s,\ell}_n)^\intercal
\eea
as the coefficient of $x^{\ell p^{s}-1}$ in the column $n$-vector of polynomials 
$\big(\frac{\Phi_s(x,z)}{x-z_1}, \dots, \frac {\Phi_s(x,z)}{x-z_n}\big)^\intercal$.
If $\ell >g$,  then the polynomial
$\frac{\Phi_s(x,z)}{x-z_i}$ does not have the monomial 
$x^{\ell p^{s}-1}$.

\vsk.2>

\begin{thm} 
[{\cite[Theorems 4.4]{V3},  \cite[Theorem 6.9]{VZ}}]
\label{thm V2}

The column $n$-vectors  $Q^{s,\ell}(z),$  $\ell=1,\dots,g$, of polynomials in $z$ are solutions of the  KZ equations
\eqref{KZ}  modulo $p^{s}$. They generate a free $(\Z/p^s)[z]$-module of rank $g$.

\end{thm}

\vsk.2>

The
 $n$-vectors  $Q^{s,\ell}(z)$, $\ell=1,\dots,g$, are called 
 the {\it $p^{s}$-hypergeometric solutions} to the KZ equations
\eqref{KZ}.

\subsection{Formal $p$-adic schemes over $\Z_p$}
\label{sec 1.4}

Recall some definitions.
A formal $p$-adic scheme $Y/\Z_p$ is a collection of schemes $Y_s/(\Z/p^s)$ together with isomorphisms
\bea
Y_{s+1} \times_ {\on{Spec}\Z/p^{s+1}} \on{Spec}\Z/p^s\simeq Y_s\,.
\eea
If $f_1,\dots,f_k\in \Z_p[z_1,\dots,z_n]$, then
\bea
 \Z_p[z_1,\dots,z_n, f_1^{-1},\dots,f_k^{-1}]^{\wedge}_p
\ :\,=\ 
\lim_{\leftarrow} (\Z/p^s)[z_1,\dots,z_n, f_1^{-1},\dots,f_k^{-1}]\,
 \eea
is the $p$-adic completion of 
$ \Z_p[z_1,\dots,z_n, f_1^{-1},\dots,f_k^{-1}]$.
The formal $p$-adic  scheme associated with $ \Z_p[z_1,\dots,z_n, f_1^{-1},\dots,f_k^{-1}]^{\wedge}_p$
is denoted by
\bea
\on{Spf} \,\Z_p[z_1,\dots,z_n, f_1^{-1},\dots,f_k^{-1}]^{\wedge}_p.
\eea
Let
$$
\frak S = \Spf \, \bZ_p [z_1,\dots, z_n, (z_i-z_j)^{-1},\; 1\leq i\ne j \leq n]^{\wedge}_p
$$
be the formal $p$-adic scheme of the  complement to the union of diagonal hyperplanes $z_i - z_j=0$, 
$1\leq i, j\leq n$.  

In what follows we consider the KZ bundle $\mc V$ over the formal $p$-adic scheme
$\frak S$ and its subschemes and reductions modulo $p^s$.

\smallskip

In  \cite[Section 5.1]{V3},\ 
 a $g\times g$-matrix $A(z)$ was defined with entries in $\Z[z]$ and  
determinant having a nonzero projection  to $\F_p[z]$. The matrix $A(z)$ projected to $\F_p[z]$
is the Hasse--Witt matrix of the curve $X(z)$ considered over $\F_p$. The fact that $\det A(z)$ is nonzero 
over $\F_p$ means that  the  curve $X(z)$ is ordinary in the terminology of \cite{DI}.

\smallskip

Let
$$
\frak D = \Spf \, \bZ_p [z_1,\dots,z_n,  (\det A)^{-1}, (z_i-z_j)^{-1},\; 1\leq i\ne j \leq n]^{\wedge}_p
$$
be the open formal $p$-adic subscheme of $\frak S$ defined by $\det A(z)\ne 0$.

 Denote 
\bean
\label{R1}
\mathcal{O}(\frak S)
&=&
\bZ_p [z_1,\dots, z_n, (z_i-z_j)^{-1},\; 1\leq i\ne j \leq n]^{\wedge}_p\,.
\\
\notag
\mathcal{O}(\frak D)
&=& 
\bZ_p [z_1,\dots,z_n,  (\det A)^{-1}, (z_i-z_j)^{-1},\; 1\leq i\ne j \leq n]^{\wedge}_p\,.
\eean

\smallskip

According to \cite[Section 6.5]{VZ}, for any $s\geq 1$,  
the column vectors
$Q^{s,\ell}(z)$, $\ell=1,\dots,g$, generate a rank $g$ flat subbundle of the KZ bundle  
$\mc V_s$ over $\frak D_s$. The subbundle is denoted by  $\mc W_s$\,.
According to \cite[Theorem 6.9]{VZ}, we have   $\mc W_{s+1}|_{\frak D_s} = \mc W_{s}$ as subbundles  of $\mc V_s$.
Thus, the sequence of bundles  $\mc W_s$ has a $p$-adic limit as $s\to\infty$,  denoted by $\mc W$, where $\mc W\subset \mc V |_{\frak D}$.

\subsection{$F$-crystal}
\label{sec 1.5}

Consider our family $X$ of hyperelliptic curves over $\frak S$.
Let $H^1_{dR}(X/\frak S)$  be the relative de Rham cohomology of the family.
 This is a bundle of rank $2g$   equipped with the Gauss-Manin  connection $\nabla^{\on{GM}}$.
We denote  it by  $\mc E$.

\smallskip

We identify $\mc V$ with $\mc E^*$ dual to $\mc E$,
cf. formula \eqref{Iga}.
  This identification allows us to identify the flat sections of the
KZ connection with the $\nabla^{\on{GM}}$-flat linear functionals $I : \mc E\to \mc O(\frak S)$.

\smallskip

The map $z_i \mapsto z_i^p$, $i=1,\dots,n$, extends to a well-defined algebra endomorphism
$F: \mc O(\frak S)\to \mc O(\frak S)$. 
Denote 
\bea
F^*\mc E = \mc O(\frak S)\ox_{\mc O(\frak S),F} \mc E .
\eea
Thus, $F^*\mc E$ is the quotient of $\mc O(\frak S)\otimes _\bZ \mc E$
by the subgroup generated by $r\otimes r'v- F(r')r\otimes v$, $r,r' \in \mc O(\frak S)$, $v\in \mc E$.
Let 
\bea
\phi : F^*\mc E \to \mc E
\eea
 be the crystalline Frobenius, see \cite{BO}. The  morphism $\phi$ respects the connections. 
 The pair $(\mc E, \phi)$ is called an $F$-crystal.

\smallskip

The $F$-crystal $(\mc E, \phi)$ restricted to $\frak D$ is ordinary in the sense of \cite{DI} since 
$\det A(z)\ne 0$ on $\frak D$.  In particular, by 
\cite{DI},  there is  a unique  exact sequence of bundles with connections over $\frak D$, 
\bean
\label{cry}
0 \to \mc U 	\hookrightarrow  \mc E \to \mc E /\mc U\to 0,
\eean
such that 
\begin{enumerate}
\item[(i)]

$\on{rank}\,\mc U=g$,

\item[(ii)]

$\phi\, \vert_{\mc F^*\mc U} : F^*\mc U\to \mc U$ is an isomorphism,

\item[(iii)]
$\phi\, \vert_{F^*(\mc E /\mc U)} 
: F^*(\mc E /\mc U)\to \mc E /\mc U$
has the form $p\cdot \phi_1$, where $\phi_1: F^*(\mc E /\mc U)\to \mc E /\mc U$  is an isomorphism, and $p\,\cdot$
\ is multiplication by $p$.

\item[(iv)]
$\mc U$ equals  $\bigcap _i \phi ^i((F^i)^* \mc E)\subset \mc E$, where $\phi^i$ is viewed as $(F^i)^* \mc E \to \mc E$.

\end{enumerate}

Note that (iv) is implied by (ii) and (iii). 

The pair $(\mc U,\phi)$ is called the unit root subcrystal of $(\mc E, \phi)$. We refer to sequence \eqref{cry} as  the slope sequence.

\smallskip

The KZ bundle $\mc V = \mc E^*$ has the dual $F$-crystal structure,
\bean
\label{dFs}
                    p\cdot (\phi^*)^{-1} : F^*(\mc E^*) \to \mc E^*,
                    \eean
denoted by $(\mc V, \psi)$. 
Here $\phi^*$ is the morphism dual to $\phi$. The restriction
of the dual $F$-crystal to $\frak D$ is ordinary.
Let 
\bean
\label{des}
0 \to \mc U' 	\hookrightarrow  \mc V \to \mc V/\mc U'\to 0
\eean
be the corresponding unique  exact sequence of bundles with connections over $\frak D$
with properties (i)-(iv). Here  $\mc U'$ is the annihilator of $\mc U$.

\subsection{$F$-crystal and $p^s$-hypergeometric sections}

Recall the $p^s$-hypergeometric solutions  $Q^{s,\ell},$  $\ell=1,\dots,g$,
and the bundle $\mc W_s$ generated by the them.

Let $\mc U'_s :=\mc U' \ox (\Z/p^s)$ be the reduction of   $ \mc U'$  modulo $p^s$.

\begin{thm}
\label{thm m1}

For $s\geq 1$, the $p^s$-hypergeometric solutions  $Q^{s,\ell},$  $\ell=1,\dots,g$, 
form a flat basis of the $\Ga(\frak D_s, \mc O_{\frak D_s})$-module   $\mc U'_s$, that is, $\mc W_s=\mc U'_s$.

\end{thm}

Theorem \ref{thm m1} is proved in Section \ref{sec pr1}.

\begin{cor}
\label{cor W=U'} 

The KZ bundle $\mc U'$ is the $p$-adic limit of the sequence of KZ bundles 
$\mc W_s$  as $s\to\infty$.
\qed

\end{cor}

Theorem \ref{thm m1} and Corollary \ref{cor W=U'} give a new proof of \cite[Theorem 6.9]{VZ} that  the sequence of bundles 
$\mc W_s$ has a $p$-adic limit as $s\to\infty$, and the limit is a KZ-invariant subbundle of $\mc V\vert_{\frak D}$. 
Notice that the proof in \cite{VZ} is based on matrix Dwork--type
congruences, while the proof of Theorem \ref{thm m1} is based on properties of ordinary $F$-crystals.

\subsection{Pairing on $\mc U'$}

We have  the Poincare pairing on  $\mc E = H^1_{dR}(X/\frak S)$ and the dual
Poincare  pairing $(\cdot,\cdot)$ on
$\mc V= \mc E^*$.  Using properties (ii), (iii), and the fact that $(\phi(a), \phi(b))=p\cdot(a,b)$ for all $a,b \in \mc V$, we conclude that  
$(\cdot,\cdot)\vert_{\mc U'} =0$ and obtain the following corollary of Theorem \ref{thm m1}.

\begin{cor} 
For   $ s\geq 1$, the $p^s$-hypergeometric solutions  $Q^{s,\ell},$  $\ell=1,\dots,g$, generate a Lagrangian subbundle of
$\mc V_s$ with respect to the Poincare pairing $(\cdot,\cdot)$.
\qed

\end{cor}

One can calculate the pairing $(\cdot,\cdot)$ explicitly.
Then equations
$(Q^{s,\ell},Q^{s,m}) = 0$, $\ell, m=1,\dots,g$, become a system of explicit congruences modulo $p^s$
 for rational functions in $z$, cf. orthogonal relations \cite[Formula (1.15)]{VV} for $p$-hypergeometric solutions.

\subsection{Local flat sections}
\label{sec local}

Here  we discuss the KZ  bundle locally. Let $\bK$ be a field of characteristic $p$, $W(\bK)$ its ring of Witt vectors. The reader may assume that $\bK=\bF_p$ and then $W(\bF_p)=\bZ_p$.

\smallskip

Choose a point $a=(a_1,\dots,a_n) \in \frak D(W(\bK))$.
 Denote  $t_i=z_i-a_i$,  $t=(t_1,\dots,t_n)$,
 \bea
 R = W(\bK)[[t_1,\dots,t_n]], \qquad  
\mc V_{\on{loc} }= \{ (v_1,\dots,v_n)\in R^{\oplus n} \mid v_1+\dots + v_n=0\}.
\eea
 Formulas \eqref{KZ} well-define a KZ connection $\nabla^{\on{KZ}}$
on $\mc V_{\on{loc} }$.  We study the flat sections of this connection.
 
\smallskip

Let $H^1_{dR}(X/R)$  be the relative de Rham cohomology of our family of hyperelliptic curves.
 This is a free $R$-module  of rank $2g$
 equipped with the Gauss-Manin  connection $\nabla^{\on{GM}}$. 
We denote it 
 by $\mc E_{\on{loc}}$.

\smallskip

We identify the $R$-module $\mc V_{\on{loc} }$ with the $R$-module $\mc E_{\on{loc} }^*$ dual to $\mc E_{\on{loc} }$
and identify
 the KZ connection $\nabla^{\on{KZ}}$ with the connection on  $\mc E_{\on{loc} }^*$ dual 
to the Gauss-Manin connection $\nabla^{\on{GM}}$.
  This identification allows us to identify the flat sections of the
KZ connection with the $\nabla^{\on{GM}}$-flat linear functionals $I : \mc E_{\on{loc} } \to R$.

\smallskip

Consider the algebra endomorphism  $F : R\to R$ defined by the formula $t_i \mapsto t_i^p$, $i=1,\dots,n$. 
Let $(\mc E_{\on{loc} }, \phi)$ be the associated $F$-crystal.
The $F$-crystal $(\mc E_{\on{loc}}, \phi)$ is ordinary since $a\in\frak D(W(\bK))$. According to \cite{DI},
there is a unique exact sequence  of $R$-modules with connections, 
\bean
\label{cryl}
0 \to \mc U_{\on{loc} } 	\hookrightarrow  \mc E_{\on{loc} } \to \mc E _{\on{loc} }/\mc U_{\on{loc} }\to 0.
\eean
such that 
\begin{enumerate}
\item[(i)]

$\on{rank}\,\mc U_{\on{loc} }=g$,

\item[(ii)]
$\phi\, \vert_{\mc F^*\mc U_{\on{loc} }} : F^*\mc U_{\on{loc} }\to \mc U_{\on{loc} }$ is an isomorphism,

\item[(iii)]
$\phi\, \vert_{F^*(\mc E_{\on{loc} } /\mc U_{\on{loc} })} 
: F^*(\mc E_{\on{loc} } /\mc U_{\on{loc} })\to \mc E _{\on{loc} }/\mc U_{\on{loc} }$
has the form $p \cdot\phi_1$, where $\phi_1: F^*(\mc E_{\on{loc} } /\mc U_{\on{loc} })\to \mc E_{\on{loc} } /\mc U_{\on{loc} }$ 
 is an isomorphism.  
\item[(iv)]
$\mc U_{\on{loc} }$ equals  $\bigcap _i \phi ^i((F^i)^* \mc E_{\on{loc} })\subset \mc E_{\on{loc} }$, 
where $\phi^i$ is viewed as $(F^i)^* \mc E_{\on{loc} } \to \mc E_{\on{loc} }$\,,

\item[(v)]
 $\mc U_{\on{loc} }$ has a $\nabla^{\on{GM}}$-flat  basis.

\end{enumerate}
Compare this list with properties (i)-(iv) in Section \ref{sec 1.5} and notice the new property (v).

\smallskip

The second main result  of this paper is the following theorem.

\begin{thm}
\label{thm m2}

${}$

\begin{enumerate}

\item[$\on{(i)}$]
If  $I : \mc E_{\on{loc} } \to R$ is a flat section of the KZ bundle, then
$I\vert_{\mc U_{\on{loc} }} = 0$.

\item[$\on{(ii)}$]  

The space of flat sections of the KZ connection on $\mc V_{\on{loc} }$ is a free $W(\bK)$-module of rank $g$.

\end{enumerate}
\end{thm}

Theorem \ref{thm m2} is 
deduced  in Section \ref{sec pr2} from 
 paper \cite{VV} where it is shown that the space of solutions of
the KZ equations \eqref{KZ} in characteristic $p$ is $g$-dimensional.
For convenience of the reader we sketch the argument from   \cite{VV} in Remark \ref{rem:sketch}.

 \begin{corr}
 
Let $I^1,\dots,I^g$  be a basis of the module of flat sections of $\mc V_{\on{loc}}$\,.
Then
there exists
a unique matrix $(b_{s,j}^i) \in
\on{Mat}_g((W(\bK)/p^s)[[t]])$,  such that $\det (b_{s,j}^i) \not\equiv 0 $ modulo $p$, and
\bean
\label{ibq}
I^i\ \equiv \ \sum_{j=1}^g\, b_{s,j}^i\, Q^{s,j} \pmod{p^s}, \qquad i=1,\dots, g. 
\eean
Moreover, each $b_{s,j}^i$ is a quasi-constant, that is $db_{s,j}^i \equiv 0 \pmod{p^s}$.
\qed

 \end{corr}

\begin{Rem}
   One can show that  for any $s>1$, there exists a flat map $\mc E_{\on{loc} } \to R/p^s$ such that 
    $I\vert_{\mc U_{\on{loc} }} \ne 0$,     cf. part (i) of Theorem \ref{thm m2}.
    That is, for any $s>1$, not every solution  to the KZ equation modulo $p^s$ 
    is a linear combination of $p^s$-hypergeometric solutions.
    
    It is shown in \cite{VV}, that every flat map $\mc E_{\on{loc} } \to R/p$
is a linear combination of    $p$-hypergeometric solutions.
\end{Rem}

\begin{Rem}
It is an interesting problem to choose a natural basis $I^1,\dots,I^g$ and describe the matrix $(b^i_{s,j})$ explicitly. Cf.
\cite[Corollary 7.4]{V2},
\cite[Theorem 9.8]{SlV},
\cite[Theorem 1.2]{SmV}.

\end{Rem}

\section{Proof of Theorem \ref{thm m1}}
\label{sec pr1}

\subsection{Some cohomology classes}

Let $\nabla^{\on{GM}}_i :=\nabla^{\on{GM}}_{\frac{\der}{\der z_i}}$
denote the differentiation with respect to the Gauss-Manin connection along the vector field 
$\frac\der{\der z_i}$.
Define the differential forms
\bea
\om_i  = \frac{dx}{(x-z_i)y}\,,\quad i=1,\dots, n.
\eea
The following facts are well-known.

\begin{lem}
\label{lem coh}

${}$

\begin{enumerate}

\item[$\on{(i)}$]
The elements  $[\om_i] \in H^1_{dR}(X/\frak S)$, $i=1,\dots, n-1$, form a basis of  
$H^1_{dR}(X/\frak S)$ and
\bean
\label{sumom}  
[\om_1]+\dots + [\om_n]=0.
\eean

\item[$\on{(ii)}$]

We have
\bean
\label{der ij}
\nabla^{\on{GM}}_i\,[\om_j] &=& - \frac 12\,\frac {[\om_i]-[\om_j]}{z_i-z_j}\,, \quad i\ne j,
\qquad
\nabla^{\on{GM}}_i\,[\om_i] = \frac 12\,\sum_{j\ne i}\,\frac {[\om_i]-[\om_j]}{z_i-z_j}\,.
\eean

\end{enumerate}

\end{lem}

\begin{cor}
\label{cor dual}
Formulas \eqref{der ij} imply the isomorphism of the KZ connection
$\nabla^{\on{KZ}}$  and the connection dual to the Gauss-Manin connection
$\nabla^{\on{GM}}$. \qed

\end{cor}

\subsection{Map $C_s$ and $p^s$-hypergeometric solutions}

For $s\geq 1$,  define a generalized Cartier map 
$C_s : \mc E \to ((\Z/p^s)[z])^g$ first on generators
$[\om_i]$, $i=1,\dots, n$, and then extend it to $\mc E$ by linearity.

We have
\bea
\om_i = \frac{dx}{(x-z)y} 
= \frac 1{y^{p^s}}\,\frac{y^{p^s-1}dx}{x-z_i}
= \frac 1{y^{p^s}}\,{\sum}_k c_{k,i}(z) x^k dx,
\eea
for suitable $c_{k,i}(z) \in \Z[z]$. Define
\bean
\label{Csom}
C_s : [\om_i] \mapsto  (c_{p^s-1,i}(z), c_{2p^s-1,i}(z), \dots, c_{gp^s-1,i}(z)).
\eean
 
Recall the $p^s$-hypergeometric solutions $Q^{s,\ell}(z) = (Q^{s,\ell}_1,\dots, Q^{s,\ell}_n)^\intercal$,
$\ell=1,\dots,g$.  Clearly
\bean
Q^{s,\ell}(z)  = (c_{\ell p^s-1,1}, c_{\ell p^s-1,2}, \dots, c_{\ell p^s-1,n}).
\eean

\begin{lem}
\label{lem cdf}
Let $\om=\sum _i g_i(z) \om_i$ be such that $\om= \frac{\der f}{\der x}(x,z)dx $ for some function $f(x,z)$, then
$C_s(\om) = 0$.

\end{lem}

\begin{proof}
By assumption we have
\bea
\frac{\der f}{\der x}(x,z)dx = \frac 1{y^{p^s}} {\sum}_{i,k} g_i(z) c_{k,i}(z) x^k dx . 
\eea
Hence
\bea
\frac{\der}{\der x}\big(y^{p^s} f(x,z))dx \equiv  {\sum}_{i,k} g_i(z) c_{k,i}(z) x^k dx \pmod{p^s}.
\eea
Thus $ {\sum}_{i} g_i(z) c_{\ell p^s-1,i}(z) \equiv 0 \pmod{p^s}$ for $\ell=1,\dots,g,$
and $C_s(\om) = 0$.
\end{proof}

\subsection{Proof of Theorem \ref{thm m1}}

Let $U\subset X$ be the affine curve defined by equation \eqref{hE}.
Set 
\bea
\mc G_s = \ker\big(H^1(X_s/\frak D_s)  \to H^1(U_s/\frak D_s)\big),
\eea
\footnote{For any nonempty affine $W\subset X$, we have
$\mc G_s = \ker\big(H^1(X_s/\frak D_s) 
\to H^1(W_s/\frak D_s)\big)$ and $\mc G_s$ does not depend on the choice of $W$.}.
By Lemma \ref{lem cdf}, we have $C_s(\mc G_s) = 0$. Hence
the generalized Cartier map $C_s$ descends to a map
\bean
\label{gcm}
C_s : H^1(X_s/\frak D_s)/ \mc G_s \to \mc O(\frak D)^g.
\eean
Since the $p^s$-hypergeometric solutions $Q^{s,\ell}$, $\ell=1,\dots, g$, are linear independent, this map is onto.
Hence, $\on{rank}\, \mc G_s \leq g$.
\footnote{Alternatively, one checks that 
$(\cdot,\cdot)\vert_{\mc G_s} =0$, where 
$(\cdot,\cdot)$ is the Poincare form, and hence, $\on{rank}\, \mc G_s \leq g$.}

\begin{lem}
We have $\mc U_s\subset \mc G_s$.
\end{lem}

\begin{proof}
Every class in  $H^1(U_s/\frak D_s)$ is represented by a global 1-form. This implies that the Frobenius on
$H^1(U_s/\frak D_s)$ is divisible by $p$. The lemma follows from (iv) in Section \ref{sec 1.5}.
\end{proof}

By (i) in Section \ref{sec 1.5}, we have $\on{rank}\,\mc U_s = g$. Hence
$\mc G_s = \mc U_s$. Theorem \ref{thm m1} is proved.

\section{Proof of Theorem \ref{thm m2}}
\label{sec pr2}

The theorem is deduced below from  paper \cite{VV} where it is shown that the space of solutions of
the KZ equations \eqref{KZ} in characteristic $p$ is $g$-dimensional. We refer the reader to Remark  \ref{rem:sketch} for a sketch 
of the relevant argument borrowed from  \cite{VV}.

\smallskip

Without loss of generality we may assume that $\K$ is algebraically closed. Denote $\Lambda= W(K)$.
Recall the $F$-crystal $(\mc V_{\on{loc}}, \psi)$ of the $R$-module $\mc V_{\on{loc}} = \mc E_{\on{loc}}^*$.
We have the corresponding slope exact sequence  of $R$-modules with connections, 
\bean
\label{crylW}
0 \to \mc U'\hookrightarrow  \mc V_{\on{loc}} 
 \to \mc V_{\on{loc}} /\mc U' \to 0.
\eean
We need to show that every flat section of $\mc V_{\on{loc}}$ lies in $\mc U'$.

\smallskip

Recall from \cite{DI} that there exists a basis $u_1,\dots,u_g$, $e_1,\dots,e_g$
of $\mc V_{\on{loc}}$ such that
\begin{enumerate}
\item[(i)]

the collection $u_1,\dots,u_g$ is a flat basis of $\mc U'$,

\item[(ii)]

for $i=1,\dots,g$, we have
\bea
\nabla^{\on{KZ}}(e_i) =\sum_{j=1}^g d\log(q_{ij}) \ox u_j\,,
\eea
where $q_{ij} \in R^\times$.
\end{enumerate}

\smallskip

Let $f=\sum_i (a_i(t) u_i + b_i(t) e_i)$ be a flat section of $\mc V_{\on{loc}}$. 
Our goal is to show that all $b_i(t)$ are equal to zero.
That statement implies the theorem.

\smallskip

Using $\nabla^{\on{KZ}}(f) =0$,  we observe that  $db_i(t)=0$
and hence   $b_i \in \Lambda$.

\smallskip

We assume that
$(b_1, b_2, \dots, b_g)= p^s (b'_1, \dots, b'_g)$,  where  the vector $(b'_1, \dots, b'_g)$
  is nonzero modulo $p$.

Writing $\nabla^{\on{KZ}}(f)=0$ we obtain
\bean
\label{da}
da_i + p^s{\sum}_{j=1}^g \, b'_j \,d\log(q_{ij}) = 0
\eean
for all $i$.  Reduce all terms of this identity modulo $p^{s+1}$. Let
\bean
\label{dan}
d \tilde a_i + p^s{\sum}_{j=1}^g\, \tilde b_j\, d\log(\tilde q_{ij}) = 0
\eean
be the resulting identity.  Here  $\tilde b_1, \dots, \tilde b_g \in \bK$.
Denote
$\eta_i=\sum_{j=1}^g \tilde b_j\, d\log(\tilde q_{ij})$.

We use the following well-known lemma (see e.g. \cite[Corollary 2.3.14 and Equation (2.2.5)]{I}).

\begin{lem}
\label{lem CD}

Let $s$ be a positive integer and  $\eta \in \Om^1(\bK[[t]])$ a closed 1-form. Identifying $\Om^1(\bK[[t]])$   with $p^s\Om^1(W_{s+1}(\bK)[[t]])$ 
we denote by  $p^s\eta$ the corresponding form in $p^s\Om^1(W_{s+1}(\bK)[[t]])$. 
Assume that $p^s\eta$ is exact. Let $C: \Om^{1}_{\text{closed}}(\bK[[t]])\to \Om^1(\bK[[t]])$ be the Cartier map. 
 Then, for each positive $i$, the 1-form
 \bea
 C^i(\eta) :=  \underbrace{C\circ \dots \circ C}_\text{i}(\eta) 
 \eea
is well-defined, closed, and $C^{s+1}(\eta) = 0$.
\qed

\end{lem}

Return to the exact  form  $p^s\eta_i$\,.\  By the lemma, we have
$C^{s+1}(\eta_i) = 0$.
Using the fact that
$C(d\log(\tilde q_{ij})) = d\log(\tilde q_{ij})$ we obtain that
\bea
C\left({\sum}_{j=1}^g\, (\tilde b_j)^{1/p^s}\, d\log(\tilde q_{ij})\right) =0.
\eea
This implies that ${\sum}_{j=1}^g\, (\tilde b_j)^{1/p^s}\, d\log(\tilde q_{ij})$ is exact.
Therefore,
${\sum}_{j=1}^g\, (\tilde b_j)^{1/p^s}\, d\log(\tilde q_{ij}) + d\tilde a_i =0$ for some 
$\tilde a_i \in \bK[[t]]$. Then
$\tilde f = \sum_i ( \tilde a_i u_i + (\tilde b_i)^{1/p^s} e_i) $ is 
a flat section of $\mc V_{\on{loc}}\ox \F_p$ with at least one nonzero
$e_i$-component.

\smallskip

This  contradicts to \cite{VV} where it is shown  that the space of flat sections of
$\mc V_{\on{loc}}\ox \bF_p$ is of rank $g$  and generated by linear combinations of $u_1,\dots,u_g$ with constant coefficients.
Theorem \ref{thm m2} is proved.
\begin{Rem} Assume that $\bK=\overline{\bF_p}$.
Then, for some $r$, we have   $C^r(\eta_i) = \eta_i$. In particular, the sequence $\eta_i, C(\eta_i), C^2(\eta_i), \cdots$ is periodic.  
On the other hand, by the lemma, we have
$C^{s+1}(\eta_i) = 0$.
We infer $\eta_i =0$ for all $i$.
Then
$\tilde f = \sum_i \tilde b_i e_i $ is 
a flat section of $\mc V_{\on{loc}}\ox \F_p$\,. 
\end{Rem}
\begin{Rem}\label{rem:sketch}
 For reader's convenience we sketch the proof, borrowed from  \cite{VV}, of the fact that the space of flat sections of
$\mc V_{\on{loc}}\ox \bF_p$ is freely generated by $p$-hypergeometric solutions (which by Theorem \ref{thm m1} form a basis for
the space of flat sections of $\mc U_{\on{loc} }'\ox \F_p $). Consider the mod $p$ reduction of sequence \eqref{cryl},
\begin{equation}
    0 \to \mc U_{\on{loc} }\ox \F_p 	\hookrightarrow  \mc E_{\on{loc} }\ox \F_p \to (\mc E _{\on{loc} }/\mc U_{\on{loc} })\ox \F_p\to 0.
\end{equation}
We need to verify that every flat map $f: \mc E_{\on{loc} }\ox \F_p \to R \ox \F_p$ vanishes on $\mc U_{\on{loc} }\ox \F_p$.
The key observation is that $f$ vanishes on the images of all the $p$-curvature operators $(\nabla^{GM}_i)^p: \mc E_{\on{loc} }\ox \F_p \to \mc U_{\on{loc} }\ox \F_p$. Thus it is enough to check that these images span $\mc U_{\on{loc} }\ox \F_p$. Using a result of Katz \cite{katz} relating the $p$-curvature operators and the Kodaira-Spencer maps the assertion follows from the following: the Kodaira-Spencer operators 
$\bar{\nabla}^{GM}_i$ applied to the $1$-form $\frac{dx}{y}$ span the space $H^1_{dR}(X/S)/F^1$. The latter can be proven directly. 

\end{Rem}

\end{document}